\documentclass[12pt,reqno]{amsart}
\usepackage{amsmath,amsthm}
\usepackage[english]{babel}
\usepackage[margin=1.5in]{geometry}
\usepackage[applemac]{inputenc}
\usepackage{tikz}
\usetikzlibrary{arrows,automata}
\usepackage{amsfonts}
\usepackage{latexsym}
\usepackage{mathtools}
\usepackage{nccmath}
\usepackage{enumerate}
\usepackage{times}
\usepackage{cases}
\usepackage{tikz-qtree}
\usepackage{tikz-cd}

\usetikzlibrary{arrows}

\tikzset{
  treenode/.style = {align=center, inner sep=1pt, text centered,
    font=\sffamily},
  arn_n/.style = {treenode, circle, white, font=\sffamily\bfseries, draw=black,
    fill=black, text width=1.5em},
  arn_r/.style = {treenode, circle, black, draw=black,
    text width=1.5em, very thick},
  arn_x/.style = {treenode, rectangle, draw=black,
    minimum width=0.5em, minimum height=0.5em}
}

\newcommand{\C}{{\mathbb C}}

\newcommand{\B}{{\mathcal B}}

\newcommand{\HH}{{\mathcal H}}

\newcommand{\lb}{\lambda}

\newcommand{\T}{\tau}

\newcommand{\pn}{{\mathbb P}^{n}}

\newcommand{{\pf}}{{\bf Proof. }}
\newtheorem{theorem}{Theorem}[section]
\newtheorem{corollary}[theorem]{Corollary}
\newtheorem{lemma}[theorem]{Lemma}
\newtheorem{proposition}[theorem]{Proposition}
\newtheorem{example}[theorem]{Example}

\newtheorem{definition}[theorem]{Definition}

\DeclareMathOperator{\re}{Re}
\DeclareMathOperator{\Sp}{Sp}
\DeclareMathOperator{\ran}{ran}

\makeatletter\@addtoreset{equation}{section} \makeatother

\usetikzlibrary{arrows}

\tikzset{
  treenode/.style = {align=center, inner sep=1pt, text centered,
    font=\sffamily},
  arn_n/.style = {treenode, circle, white, font=\sffamily\bfseries, draw=black,
    fill=black, text width=1.5em},
  arn_r/.style = {treenode, circle, black, draw=black,
    text width=1.5em, very thick},
  arn_x/.style = {treenode, rectangle, draw=black,
    minimum width=0.5em, minimum height=0.5em}
}

\begin{document}

\title{Joint spectrum, Group representations, and Julia Set}

\author[R. Yang]{Rongwei Yang}
\address{Rongwei Yang: Department of Mathematics and Statistics, University at Albany, the State University of New York,
Albany, NY 12222, U.S.A.}
 \email{ryang@albany.edu}

\maketitle

\footnotetext{
\hangindent=1.2em\emph{Key words and phrases}: projective spectrum, dihedral group, $C^*$-algebra, weak containment, self-similar representation, Fatou set, Julia set}

\begin{abstract}
This mostly expository note describes a newly discovered link between joint spectrum, group representations, and multivariable complex dynamics. The first part gives a brief review of some notions of joint spectrum for linear operators, focusing on their connections. The second part contains some new results. It starts with a multivariable characterization of weak containment and amenablility. Then it revisits an application of projective spectrum to the study of self-similar group representations made in \cite{GY}. It shows that the projective spectrum of the infinite dihedral group $D_\infty$ coincides with the Julia set of a rational map $F_\pi: \mathbb{P}^2\to \mathbb{P}^2$ derived from the self-similarity of $D_\infty$. This result improves the main theorem in \cite{GY}, and it provides a rare example where the Julia set of a nontrivial multivariable map can be explicitly described.
\end{abstract}

\section{Introduction}

Consider a complex separable Hilbert space $\HH$. We denote by $B(\HH)$ the set of bounded linear operators on $\HH$. For an element $A\in B(\HH)$, its spectrum $\sigma(A)$ holds important information about $A$ and thus has been extensively studied in operator theory. For several operators $A_1, ..., A_n\in B(\HH)$, however, it is not at all clear how to make a good definition of their joint spectrum. The problem is nontrivial even for $2\times 2$ matrices. Ideally, a good definition of joint spectrum should have the following properties.

 {\bf 1}. It is a natural extension of the classical spectrum $\sigma(A)$. 
 
 {\bf 2}. It is easy to compute in a plethora of examples.
 
 {\bf 3}. It reveals the algebraic connections among the operators, if there are any.
 
 {\bf 4}. It reflects joint behaviors of the operators as a group.

\noindent The search for joint spectrum started in the late 1960s. An extensive effort has been made during the 70s and the 80s which laid the foundation for multivariable operator theory. Indeed, joint spectrum remains a main theme in the theory to this day, with deep results and far-reaching connections with a wide range of fields in mathematics, science, and engineering. The first part of this note provides a brief review of this development, focusing on the connections among various notions of joint spectrum. Some short proofs are given for depth and clarity. The second part describes some applications of projective spectrum to group representation theory. Some new results are given here. In particular, a link between projective spectrum and the Julia set is unambiguously established.

\part{Joint Spectra}

When the operators are commuting, i.e., $A_iA_j=A_jA_i, 1\leq i, j\leq n$, the situation is more tractable, and several notions of joint spectrum have been proposed and well-studied. In this case, one may consider the problem in an abelian unital Banach subalgebra $\B\subset B(\HH)$ containing the operators that is inversion-closed in the sense that if $a\in \B$ is invertible in $B(\HH)$, then its inverse $a^{-1}\in \B$. A bounded linear functional $m$ in the dual space $\B^*$ is said to be {\em multiplicative} if it is an algebra homomorphism from $\B$ to $\C$. The set of multiplicative linear functionals on $\B$ is denoted by $M_\B$. Due to the Banach-Alaoglu theorem, $M_\B$ is compact with respect to the $weak^*$ topology on $\B^*$.

\section{The Spectra of H\"{o}rmander, Taylor, and Harte}

Three influential notions of joint spectrum have been introduced in the 1970s. The first one is studied in H\"{o}rmander \cite{Ho}.
\begin{definition}\label{jsp}
For a tuple $A=(A_1, ..., A_n)$ of elements in a unital abelian Banach algebra ${\mathcal B}$, their joint spectrum Sp$(A)$ is the collection of $\lb=(\lb_1, ..., \lb_n)\in \C^n$ such that the ideal generated by $A_1-\lb_1I, ..., A_n-\lb_n I$ is proper in $\B$.
\end{definition}
\noindent In other words, $\lb$ is not in Sp$(A)$ if and only if there are elements $B_1, ..., B_n$ in $\B$ such that $(A_1-\lb_1I)B_1+\cdots +(A_n-\lb_nI)B_n=I$. The following theorem gives an explicit description of Sp$(A)$.

\begin{theorem}
For arbitrary elements $A_1, ..., A_n$ in a unital abelian Banach algebra ${\mathcal B}$, we have
$\Sp(A)=\{\left(m(A_1), ..., m(A_n)\right)\in \C^n\mid m\in M_\B\}$.
\end{theorem}

The second notion of joint spectrum, on the first sight, is less intuitive. It is defined in Taylor \cite{Ta} and has later become one of the main subjects in multivariable operator theory. Let $\{e_1, ..., e_n\}$ be a basis of a complex vector space $V$. The wedge product $\wedge$ on $V$ is a bilinear operation such that
\begin{equation}\label{wedge}
u\wedge v+v\wedge u=0, \ \ \ u, v\in V.
\end{equation}
For $1\leq p\leq n$, the space of $p$-forms, denote by $\Lambda^p(V)$, is spanned by 
$e_{i_1}\wedge \cdots \wedge e_{i_p}, 1\leq i_1<\cdots <i_p\leq n$. It follows from (\ref{wedge}) that 
$\Lambda^p(V)=\{0\}$ for $p>n$. The space $V$ itself is denoted by $\Lambda^0(V)$. 
The direct sum $\oplus_{0\leq p\leq n} \Lambda^p(V)$ is called the exterior algebra of $V$. For commuting operators $A=(A_1, ..., A_n)$ on a Hilbert space $\HH$, their Taylor spectrum is defined through the following Koszul complex $E(\HH, A)$:
\begin{equation}\label{koszul}
  0 \stackrel{d_{-1}}{\longrightarrow} \HH\otimes \Lambda^0 \stackrel{d_0}{\longrightarrow} \HH\otimes \Lambda^1 \stackrel{d_1}{\longrightarrow} \cdots \stackrel{d_{n-1}}{\longrightarrow} \HH\otimes \Lambda^n \stackrel{d_n}{\longrightarrow} 0,
\end{equation}
where $d_{-1}=d_n=0$, and $d_p: \HH\otimes \Lambda^p(V)\to \HH\otimes \Lambda^{p+1}(V), 0\leq p\leq n-1,$ is induced by the map
\[d_p (x\otimes \omega)=\sum_{i=1}^nA_ix\otimes (e_i\wedge \omega),\ \ \ x\in \HH,\ \omega\in \Lambda^p(V).\]
The commutativity of the operators implies that $d_{p+1}d_p=0$, or in other words, \[\ran d_p\subseteq \ker d_{p+1}, \ \ \ 0\leq p\leq n-1.\] The complex $E(\HH, A)$ is said to be {\em exact} if $\ran d_p=\ker d_{p+1}$ for every $p$. For a vector $\lb=(\lb_1,..., \lb_n)\in \C^n$, the shifted tuple $(A-\lb_1I, ..., A_n-\lb_n I)$ is denoted simply by $A-\lb$. 
\begin{definition}
The {\em Taylor spectrum} of the tuple $A$ is defined as 
\[\sigma_T(A)=\{\lb\in \C^n\mid E(\HH, A-\lb)\ \text{is not exact}\}.\]
\end{definition}
\noindent The readers will gain some intuition by considering the case $n=1$. In general, it can be shown that $\sigma_T(A)$ is a nontrivial compact subset of $\C^n$. A connection between the two joint spectra above is as follows.
 \begin{proposition}
 Let $A=(A_1, ..., A_n)$ be operators in an abelian Banach subalgebra $\B\subset B(\HH)$. Then $\sigma_T(A)\subseteq \Sp(A)$.
 \end{proposition}
 \begin{proof}
 We show that $\Sp^c(A)\subset \sigma^c_T(A)$. Without loss of generality, we prove for the case $\lb=0$, namely, if $0\notin \Sp(A)$ then $0\notin \sigma_T(A)$. 
 Assume $B_1, ..., B_n$ are elements in $\B$ such that $A_1B_1+\cdots +A_nB_n=I$.
 We set $t_0=0$ and for $1\leq p\leq n$ define $t_p: \HH\otimes \Lambda^p\to \HH\otimes \Lambda^{p-1}$ by 
 \[t_p(x\otimes e_{i_1}\wedge \cdots \wedge e_{i_p})=\sum_{m=1}^p(-1)^{m-1} B_{i_m}x\otimes e_{i_1}\wedge \cdots \wedge \widehat{e_{i_m}}\wedge \cdots \wedge e_{i_p},\]
 where $\hat{a} $ stands for the omission of $a$. Then a direct computation verifies that $t_1d_0 (x)=x$, and for $1\leq p\leq n-1$ we have
 \[(t_{p+1}d_p+d_{p-1}t_p)(x\otimes e_{i_1}\wedge \cdots \wedge e_{i_p})=\sum_{j=1}^nA_jB_j x\otimes e_{i_1}\wedge \cdots \wedge e_{i_p}=x\otimes e_{i_1}\wedge \cdots \wedge e_{i_p}.\]
 Since all involved maps are linear, we have 
 \begin{equation}
 t_{p+1}d_p(h)+d_{p-1}t_p (h)=h,\ \ h\in \HH\otimes \Lambda^p.
 \end{equation} 
 Consequently, if $d_p(h)=0$, then $h\in \ran d_{p-1}$ for each $0\leq p\leq n$. This shows that $E(\HH, A)$ is exact.
 \end{proof}  
Let $\Omega\subset \C^n$ be a bounded domain that contains $\sigma_T(A)$. Then the evaluation $f(A)$ is well-defined through functional calculus. A more intuitive way of defining $f(A)$ is to replace the variables $z_i$ by $A_i$ in the power series expansion of $f$. However, this method requires an argument on the convergence of the new series.
The following spectral mapping theorems hold for Taylor spectrum.  
\begin{theorem}\label{specm}
For every commuting $n$-tuple $A$ and function $f$ holomorphic on a neighborhood of $\sigma_T(A)$, we have $\sigma_T(f(A))=f(\sigma_T(A))$.
\end{theorem}
 
Unlike the two spectra above, the notion of {\em Harte spectrum} introduced in \cite{Ha72} is valid for noncommuting operators.
 \begin{definition}
For a tuple $A$ of operators on $\HH$, the joint spectrum  $\sigma_H(A)$ is the set of vectors $\lb\in \C^n$ such that at least one of the following two conditions holds: 

1) there exists a sequence of unit vectors $x_k\in \HH$ such that 
 \begin{equation}\label{bbelow}
 \lim_{k\to \infty} \|(A_j-\lb_j)x_k\|=0,\ \forall 1\leq j\leq n;
 \end{equation}
 
2) the sum of vector spaces $(A_1-\lb_1)\HH+\cdots +(A_n-\lb_n)\HH$ is not equal to $\HH$. 
\end{definition}

The relation between Harte spectrum and Taylor spectrum is described as follows.
 \begin{proposition}
 For a commuting tuple of bounded linear operators $A=(A_1, ..., A_n)$ on $\HH$, we have $\sigma_H(A)\subseteq \sigma_T(A)$.
 \end{proposition}
 \begin{proof}
 Without loss of generality, we show that if the $0$ vector is in $\sigma_H(A)$ then it is in $\sigma_T(A)$. According to the definition, there are three cases when $0\in \sigma_H(A)$.
 
 i) There exists a nontrivial vector $x\in \HH$ such that $A_ix=0$ for all $1\leq i\leq n$. Clearly, this implies $d_0(x)=0$ and hence $x\in \sigma_T(A)$.
 
 ii) It holds that $\cap _{i=1}^n \ker A_i=\{0\}$, but there exists a sequence of unit vectors $x_k\in \HH$ such that (\ref{bbelow}) holds for $\lambda=0$. Then we must have $\ran d_0\neq \ker d_1$ in this case. Otherwise, $\ran d_0$ would be closed, and it would follow that the map $d_0: \HH\otimes \Lambda^0\to \ran d_0$ is invertible, which contradicts (\ref{bbelow}) for $\lb=0$ because 
 \begin{align*}
 \|d_0 x\|^2=\|A_1 x\otimes e_1+\cdots +A_n x\otimes e_n\|^2=\|A_1 x\|^2+\cdots +\|A_n x\|^2
\end{align*}
is bounded below by a positive constant for all unit vector $x$.

iii) The space $A_1\HH+\cdots +A_n\HH$ is not equal to $\HH$. This implies that $d_{n-1}$ is not surjective.

\noindent In all three cases the Koszul complex $E(\HH, A)$ is not exact, and therefore $0\in \sigma_T(A)$.
\end{proof}

If the first condition in (\ref{bbelow}) is met, then the vector $\lb$ is said to be in the {\em joint approximate point spectrum} of $A$ denoted by $\sigma_\pi(A)$ (Dash \cite{Da73}). For a tuple of normal operators, the following fact is due to Cho-Takaguchi \cite{CT82}, and its proof is elegant.
\begin{theorem}
If $A$ is a tuple of commuting normal operators, then $\sigma_\pi(A)=\Sp(A)$.
\end{theorem}
\begin{proof}
First of all, the Fuglede theorem (\cite{Fu}) states that if $T$ and $N$ are commuting operators on $\HH$ and $N$ is normal, then $T$ also commutes with $N^*$. Since the $A_j$s are commuting and are all normal, we have $A_iA_j^*=A_j^*A_i$ for all $1\leq i, j\leq n$. Thus the $C^*$-algebra $\B$ generated by $I, A_1, ..., A_n$ is abelian. Moreover, it is known that every unital $C^*$-algebra is inversion-closed (\cite{Do}).

It is clear that $\sigma_\pi(A)\subseteq \Sp(A)$. To show the inclusion in the other direction, without loss of generality we show that if $0\notin \sigma_\pi(A)$ then $0\notin \Sp(A)$. In this case, there is a constant $\alpha>0$ such that $\sum_{j=1}^n\|A_jx\|^2\geq \alpha \|x\|^2, j=1, ..., n$ for all nonzero $x$, which implies that $\sum_{j=1}^nA_j^*A_j$ is bounded below and hence invertible.  If we set \[B_k=A_k^*\big(\sum_{j=1}^nA_jA_j^*\big)^{-1}, k=1, ..., n,\] then $B_k\in \B$ and $A_1B_1+\cdots +A_nB_n=I$. This shows $0\notin \Sp(A)$.
\end{proof}

From the aforementioned facts, it follows an appealing identity about normal operators. 
\begin{corollary}
If $A$ is a tuple of commuting normal operators on a Hilbert space $\HH$, then 
$\sigma_\pi(A)=\sigma_H(A)=\sigma_T(A)=\Sp(A)$.
\end{corollary}

\section{Projective Spectrum} 

Unfortunately, the joint spectra Sp$(A)$ and $\sigma_T(A)$ don't have a natural generalization to noncommuting operators. While the Harte spectrum is defined for general operators, it is usually difficult to compute in the noncommuting case. These challenges have motivated a drastically different approach. In \cite{Ya}, the notion of projective spectrum was introduced as follows.
\begin{definition}\label{proj}
 For several elements $A_0, A_1, ..., A_n$ in a unital Banach algebra ${\mathcal B}$, their projective spectrum $p(A)$ is the collection of $z\in \pn$ such that the multiparameter linear pencil $A(z)=z_0A_0+z_1A_1+\cdots +z_nA_n$ is not invertible in ${\mathcal B}$.
 \end{definition}
The letter ``$A$'' in ``$p(A)$'' refers to the pencil $A(z)$. Apparently, this definition does not take the unit $I$ as a base point, and it treats the linears operators in an equal footing. This conspicuously simple and symmetric definition enables the computation of many examples, and it has led to an extensive spectral theory with far-reaching connections. The upcoming book \cite{Ya24} will give a detailed account of this development. The following fact is a consequence of the Hartogs extension theorem in several complex variables.

\begin{theorem}
For any elements $A_0, A_1, ..., A_n$ in a unital Banach algebra ${\mathcal B}$, their projective spectrum $p(A)$ is a nonempty compact subset of $\pn$.
\end{theorem}

In the case the operators are commuting, it is a natural question whether or how the projective spectrum $p(A)$ is related to the Taylor spectrum $\sigma_T(A)$. 
\begin{proposition}\label{abelian}
Let $A_0, A_1, ..., A_n$ be commuting operators on ${\mathcal H}$. Then

 a) $p(A)=\cup_{w\in \sigma_T(A)}H_{w}$, where $H_w$ is the projective hyperplane \[\{z\in \pn \mid z_0w_0+\cdots +z_nw_n=0\};\]
 
 b) if $f(z)=(f_1(z), ..., f_m(z))$ is a vector-valued holomorphic function on an open neighborhood of $\sigma_T(A)$, then $p(f(A))=\cup_{w\in \sigma_T(A)}H_{f(w)}$.
\end{proposition}
\begin{proof}
The proof is a direct application of the spectral mapping theorem (Theorem 2.5). Indeed, if we consider the linear function \[f_z(w)=z_0w_0+\cdots +z_nw_n,\ \ \ z, w\in \C^{n+1},\] then $A(z)=f_z(A)$, and hence 
$\sigma(A(z))=f_z(\sigma_T(A)),\ z\in \C^{n+1}$. This implies that $A(z)$ is not invertible if and only if there exists a $w\in \sigma_T(A)$ such that $f_z(w)=0$, i.e., $z\in H_w$. To prove b), first note that the tuple $f(A)=(f_1(A), ..., f_m(A))$ is well-defined by functional calculus, and it is commuting. The claim then follows from Theorem \ref{specm} and part a).
\end{proof}
 Proposition \ref{abelian} b) can be viewed as the spectral mapping theorem for projective spectrum.
Remarkably, the converse of Proposition \ref{abelian} a) is also true in some cases. The following theorem is due to Chagouel-Stessin-Zhu \cite{CSZ}. The proof here is simplified.

\begin{theorem}\label{normal}
Let $A_1, ..., A_n$ be normal $k\times k$ matrices. Then they pairwise commute if and only if the projective spectrum of the tuple $(I, A_1, ..., A_n)$ is a union of projective hyperplanes in $\pn$.
\end{theorem}
\begin{proof}
In view of Proposition \ref{abelian} a), it only remains to prove the sufficiency. Since we need to prove that the matrices pairwise commute, without loss of generality we do so for the case $n=2$. Consider the homogeneous polynomial $Q_A(z)=\det (z_0+z_1A_1+z_2A_2)$. The assumption indicates that it is a product of linear factors. We shall first show that $A_1$ and $A_2$ share a common eigenvector.

Since $A_2$ is normal, we can pick an eigenvalue $\mu$ of $A_2$ such that $|\mu|=\|A_2\|$.
Then the point $(\mu, 0, -1)$ is a zero of $Q_A$. Hence, one factor of $Q_A$ must be of the form $L(z)=z_0+\lb z_1+\mu z_2$, where the coefficient $\lb$ is in fact an eigenvalue of $A_1$. 

Pick any sequence of nonzero complex numbers $\{\epsilon_s\mid s=1, 2, ... \}$ convergent to $0$ and set $\lb_{\epsilon_s}=\lb+\epsilon_s \mu$ and $A_{\epsilon_s}=A_1+\epsilon_s A_2$.
Then $(\lb_{\epsilon_s}, -1, -\epsilon_s)$ is a zero of $L(z)$, and hence $\lb_{\epsilon_s}-A_{\epsilon_s}$ is not invertible. We let $P_s$ be the orthogonal projection from $\C^k$ onto $\ker (\lb_{\epsilon_s}-A_{\epsilon_s})$ and pick any $v_s\in \ker (\lb_{\epsilon_s}-A_{\epsilon_s})$ such that $\|v_s\|=1$.
Since the unit sphere of $\C^k$ is compact, the sequence $\{v_s\}$ contains a convergent subsequence. Without loss of generality, we assume $\{v_s\}$ itself is convergent to $v$. It is clear that $v$ is an eigenvector of $A_1$. To proceed, we observe that $P(\lb-A_1)=(\lb-A_1)P=0$. Therefore, multiplying the equations
\begin{align*}
0=(\lb_{\epsilon_s}-A_{\epsilon_s})v_s=(\lb-A_1)v_s+\epsilon_s(\mu-A_2)v_s
\end{align*}
by $P$ on the left, we obtain $P(\mu-A_2)v_s=0$ for each $s$, and consequently $P(\mu-A_2)v=0$. Since $Pv=v$ and $\|v\|=1$, it follows that
\[\mu-\langle A_2 v, v\rangle=\langle (\mu-A_2)v, Pv\rangle=\langle P(\mu-A_2)v, v\rangle=0.\]
Since $A_2$ is normal and $|\mu|=\|A_2\|$, the Cauchy-Schwarz inequality implies that $A_2v$ is a scalar multiple of $v$, e.g., $v$ is an eigenvector of $A_2$ with corresponding eigenvalue $\mu$.

Since $A_1$ and $A_2$ have a common eigenvector $v$, with respect to the orthogonal decomposition $\C^k=\C v\oplus (\C v)^\perp$, they are of the forms
\[\begin{pmatrix}
\lb & 0\\
0 & A_1'\end{pmatrix}\hspace{6mm} \text{and} \hspace{6mm}  \begin{pmatrix}
\mu & 0\\
0 & A_2' \end{pmatrix},\]
respectively, where $A_1'$ and $A_2'$ are normal $(k-1)\times (k-1)$ matrices whose joint characteristic polynomial is $Q_A(z)/L(z)$. 
Then we may repeat the above arguments for $A_1'$ and $A_2'$ and, after $k-1$ steps, obtain a simultanuous diagonalization of $A_1$ and $A_2$. This concludes the proof.
\end{proof}
\noindent Determining whether two normal matrices commute is often a painstaking work. Theorem \ref{normal} provides an elegant approach. A generalization of this result to compact normal operators was later given in Mao-Qiao-Wang \cite{MQW}.

\part{Applications to Group Representations}

The simplicity and symmetry of Definition \ref{proj} enables us to compute projective spectra in a plethora of examples. A particularly interesting case is when the operators have some algebraic relations, for example, when they generate a group. Consider a finitely generated group $G=\langle g_1, g_2, \cdots, g_n\rangle$ and a unitary representation $\pi$ on a Hilbert space $\HH)$. Since the unit $1$ is a special element in $G$, it makes good sense to consider the pencil $A_\pi(z)=z_0I+z_1\pi(g_1)+\cdots + z_n\pi(g_n)$.
Although projective spectrum $p(A_\pi)$ depends on the choice of the generating set $\{g_1, ..., g_n\}$, it is capable of capturing intrinsic properties of $G$ and $\pi$.
If $\pi$ is a finite dimensional representation, then the determinant $Q_\pi (z)=\det A_\pi (z)$ is called the characteristic polynomial of $G$ with respect to $\pi$, in which case $p(A_\pi)$ is the zero set of $Q_\pi$. Study of $Q_\pi$ for finite groups can be traced back to the work of Dedekind and Frobenius in the late 19th century (\cite{De,Fr}). In fact, their work is the starting point of group representation theory (\cite{Di75}). More recent work on the joint spectrum of groups can be found in \cite{BG,GS}. Since the definition of projective spectrum, some new lines of research on group representations have emerged in \cite{CST,GY17,GY}. The second part of this note concerns mostly with the left regular representation and Koopman representation. Some new results will be reported here. 

\begin{definition} If a locally compact group $G$ has a measure-preserving action on a measure space $(X, \mu)$, then the {\em Koopman representation} $\pi: G\to B\big(L^2(X, \mu)\big)$ is defined by 
$\pi(g)f(x)=f(g^{-1}x),\ x\in X, g\in G$.
\end{definition}
\noindent In the case $X=G$ and $\mu$ is the Haar measure, the Koopman representation is the {\em left regular representation} of $G$, and it is often denoted by $\lb_G$ or simply $\lb$.

\section{Examples and Main Theorems}

We first give three examples that are relevant to the study here.

\begin{example}\label{Hs}
If $1_G$ is the trivial representation of $G$, then $A_{1_G}(z)=z_0+\cdots +z_n$ and hence $p(A_{1_G})$ is the hyperplane \[H_0:=\{z\in \pn\mid z_0+\cdots +z_n=0\}.\]
\end{example}

\begin{example}\label{specD}
The infinite dihedral group $D_{\infty}=\langle a, t \mid a^2=t^2=1\rangle$ is isomorphic to the free product ${\mathbb Z}_2\ast {\mathbb Z}_2$. Although it is probably the simplest nonabelian group, it plays an important role in several areas of mathematics. We consider the pencil $A_\lb(z)=z_0I+z_1\lb(a)+z_2\lb(t)$. It is shown in \cite{GY17} that
\[p(A_\lb)=\bigcup_{-1\leq x\leq 1}\{z\in \mathbb{P}^2\mid z_0^2-z_1^2-z_2^2-2z_1z_2x=0\}.\]
\end{example}

\begin{example}\label{F2}
Consider the free group $F_{n}=\langle g_1, ..., g_{n}\rangle$ with $n\geq 2$. It follows from \cite{BCY} Proposition 3.2 that the pencil 
$A_\lb(z)=z_0I+z_1\lb(g_1)+\cdots +z_{n}\lb(g_n)$ has projective spectrum $p(A_\lb)=\cap _{j=0}^{n}R _{j}$,
where
\begin{equation*}
R _{j}=\{z\in \pn \mid 2|z_{j}|^{2}\leq \|z\|^{2}\},\ \ \ j=0, 1, ..., n.
\end{equation*}
To be more explicit, for the case $n=2$ the spectrum $p(A_\lb)$ is equal to
\begin{align*}\label{F2}
\{|z_0|^2\leq |z_1|^2+ |z_2|^2\}\cap \{|z_1|^2\leq |z_0|^2+ |z_2|^2\} \cap \{|z_2|^2\leq |z_0|^2+ |z_1|^2\}.
\end{align*}
\end{example}

In Section 5, based on the connection between amenability and {\em weak containment}, we prove the following result.
\begin{theorem}\label{main1}
A finitely generated group $G=\langle g_1, \cdots, g_n\rangle$ is amenable if and only if $H_0\subset p(A_\lb)$.
\end{theorem}
\noindent Similar descriptions of Haagerup property and Kazhdan's property (T) of groups are also given.

For the infinite dihedral group $D_\infty$, it is known that its Koopman representation $\pi$ is weakly equivalent to its regular representation, and it follows that $p(A_\pi)=p(A_\lb)$ (\cite{GY17}). Moreover,
$\pi$ is {\em self-similar}, and this fact gives rise to a rational map $F_\pi(z)=[2\T(z)z_0:z_1:2\T(z) z_2 +z_1]$ on $\mathbb{P}^2$, where $\tau: \mathbb{P}^2\to \hat{\C}=\C\cup \{\infty\}$ is defined by 
\begin{equation}\label{eq: tauhat}
\tau(z) = \begin{dcases}
      0, & \text{ if } z_0^2-z_1^2-z_2^2=0; \\
      \dfrac{z_0^2-z_1^2-z_2^2}{2z_1z_2}, &  \text{ otherwise. }
    \end{dcases}
\end{equation}
\noindent 
Observe that $z\in p(A_\pi)$ if and only if $\tau(z)\in [-1, 1]$ (Example \ref{specD}). Section 6 studies the {\em Julia set} ${\mathcal J}(F_\pi)$ of the map $F_\pi$ and prove the following striking fact. 
\begin{theorem}\label{main2}
${\mathcal J}(F_\pi)=p(A_\pi).$
\end{theorem}
\noindent This improves the main result in \cite{GY}. It is a rare example of a nontrivial rational map in several variables whose Julia set can be completely described.

\section{Amenability}

The Banach-Tarski paradox asserts that there exists a decomposition of a solid ball in ${\mathbb R}^3$ into several disjoint pieces such that they can be reassembled, by means of rotation and shift, to form two solid balls identical to the original one. The construction of the decomposition is achieved through a decomposition of the free group $F_2$. But lying at the core of this unintuitive phenomenon is the Axiom of Choice which allows for the construction of nonmeasurable sets. The notion of amenability was introduced by von Neumann \cite{Neu} to describe groups that do not give rise to the Banach-Tarski paradox. A finitely generated group $G=\langle g_1, \dots, g_n\rangle$ is said to be {\em amenable} if there exists a {\em $G$-invariant mean}, namely a state on the $C^*$-algebra $L^\infty(G)$ satisfying \[\phi(_gf)=\phi(f),\ \ \ f\in L^\infty (G), g\in G,\] where $_gf(x)=f(g^{-1}x), x\in G$.
Clearly, every compact group is amenable because $\phi$ defined by
\[\phi (f)=\int_G fd\mu,\]where $\mu$ is the normalized Haar measure on $G$, is an invariant mean. Moreover, due to the Markov-Kakutani Theorem, every abelian group is amenable. Amenability has several equivalent characterizations, one of which involves the notion of weak containment of unitary group representations.

\begin{definition}
Consider two representations $(\pi, \HH)$ and $(\rho, \mathcal{K})$ of a discrete group $G$. We say $\pi$ is weakly contained in $\rho$ (denoted by $\pi\prec \rho$) if for every $x\in {\mathcal H}$, every finite subset $F\subset G$, and every $\epsilon>0$, there exist $y_1, \dots, y_n$ in ${\mathcal K}$ such that for all $g\in F$, we have
$|\langle \pi(g)x,\ x\rangle-\sum_{i=1}^{n}\langle \rho(g)y_i,\ y_i\rangle|<\epsilon$.
\end{definition}
\noindent Moreover, two representations $\pi$ and $\rho$ are said to be weakly equivalent if $\pi\prec \rho$ and $\rho\prec \pi$, in which case we write $\pi\sim \rho$. Weak containment provides a partial order on the set of unitary representations of $G$, and there are several equivalent statements. In particular, for finite groups $\pi\prec \rho$ if and only if $\pi$ is contained in $\rho$, i.e., $\pi<\rho$. We refer readers to \cite{BHV, Di69} for details. 

For a unitary representation $(\rho, \HH)$ of $G=\langle g_1, \dots, g_n\rangle$, we let $C^*_{\rho}(G)$ denote the $C^*$-algebra generated by the unitaries $\rho(g_i), 1\leq i\leq n$. 
Let $\mu$ be a positive regular Borel measure on $G$, and let $L^1(G, \mu)$ be the set of integrable complex functions on $G$. For a unitary representation $\pi$, we define
\[\pi(m)=\int_Gm(g)\pi(g)d\mu(g),\]
where the convergence of the integral is with respect to the weak operator topology. Since each $\pi (g)$ is a unitary, we have
\[\|\pi(m)\|\leq \int_G|m(g)|d\mu(g)=\|m\|_1.\] Further, if the the constant function $1\in L^1(G, \mu)$, we write $\pi(1)$ as $\pi(\mu)$ to avoid confusion. Clearly, for a discrete group $G$ with counting measure and an arbitrary element $m=z_1g_1+\cdots+z_kg_k\in \C[G]$, we have $\pi(m)=z_1\pi(g_1)+\cdots+z_k\pi(g_k)$. Other measures on a countable group is also useful.

\begin{example}\label{Markov}
Consider the measure $\mu$ on group $G$ generated by the set $S=\{g_1, ..., g_n\}$ such that $\mu\{g_i\}=\frac{1}{n},\ j=1, \dots, n$, and $\mu\{g\}=0$ for all $g\notin S$. Clearly, $\mu$ is a probability measure that is absolutely continuous with respect to the Haar measure on $G$, and $\mu$'s support is $S$. In this case
\[\pi(\mu)=\int_G\pi (g)d\mu(g)=\frac{1}{n}(\pi(g_1)+\pi(g_2)+\cdots +\pi(g_n)),\]
and it is called the {\em Markov operator} with respect to the generating set $S$ and the representation $\pi$, for which we denote by $M_{\pi}$. Since each $\pi(g_i)$ is a unitary, the triangular inequality implies $\|M_{\pi}\|\leq 1$.
\end{example}
\noindent As we shall soon see, the spectrum of $M_{\pi}$ holds important information about amenability. Weak containment has a useful description in terms of $C^*$-algebras.

\begin{proposition}\label{weakp}
Let $\rho$ and $\pi$ be two unitary representations of a countable group $G$. Then the following statements are equivalent:

(a) $\pi\prec \rho$;

(b) $\sigma(\pi(m))\subset \sigma(\rho(m))$ for each $m\in \ell^1(G)$;

(c) $\|\pi(m)\|\leq \|\rho(m)\|$ for each positive $m\in \C[G]$;

(d) the map $\phi: C^*_\rho(G)\longrightarrow C^*_\pi(G)$ defined by $\phi(\rho(g))=\pi(g),\ g\in G$ extends to a surjective homomorphisim.
\end{proposition}
The following theorem is \cite{BHV} Theorem G 3.2.
\begin{theorem}[Hulanicki-Reiter]\label{HR} Let $G$ be a locally compact group. The following properties are equivalent:

(i) $G$ is amenable;

(ii) ${1}_G \prec \lb$;

(iii) $\pi \prec \lb$ for every unitary representation $\pi$ of $G$.
\end{theorem}
\noindent In other words, the group $G$ is amenable if and only if its left regular representation is maximal with respect to weak containment. A spectral characterization of amenability was given by Kesten \cite{Ke1}.

\begin{theorem}\label{Kesten}
Let $G$ be a locally compact group and $\mu$ be a probability measure on $G$. Assume that $\mu$ is absolutely continuous with respect to the Haar measure on $G$ and $supp(\mu)$ generates a dense subgroup of $G$. Then the following are equivalent:

(a) $G$ is amenable.

(b) $1\in \sigma(\lambda(\mu))$.

(c) The spectral radius of $\lambda(\mu)$ is $1$.
\end{theorem}
\noindent If $\mu$ is the probability measure in Example \ref{Markov}, then $\lb(\mu)=M_\lb$. Hence a finitely generated group $G$ is amenable if and only if $1\in \sigma(M_{\lambda})$.

\subsection{Throught the Lense of Projective Spectrum}

Two representations $(\pi, {\mathcal H})$ and $(\rho, {\mathcal K})$ are said to be equivalent if there is a unitary map $U:\ {\mathcal H}\to {\mathcal K}$ such that $\rho(g)=U\pi(g)U^{-1}, g\in G$. It is obvious that in this case for any elements $g_1, ..., g_n\in G$, we have $A_{\pi}(z)=UA_{\rho}(z)U^{-1}$ and hence $p(A_{\pi})=p(A_{\rho})$. This indicates that projective spectrum is a unitary invariant for group representations. Moreover, if $\pi\prec \rho$, then Proposition \ref{weakp} (d) implies that, for each $m\in \ell^1(G)$, if $\rho(m)$ is invertible in $C^*_{\rho}(G)$ then $\pi(m)$ is invertible in $C^*_{\pi}(G)$. In particular, for $m(z)=z_01+z_1g_1+ \cdots +z_ng_n\in \C[G]$, we have $A_\pi(z)=\pi(m(z))$ and $A_\rho(z)=\rho(m(z))$. The following is thus immediate.
\begin{lemma}\label{weakeq}
Consider two unitary represntations $\pi$ and $\rho$ of group $G=\langle g_1, ..., g_n\rangle$. If $\pi\prec \rho$, then $p(A_{\pi})\subset p(A_{\rho})$.
\end{lemma}
\noindent It follows that if $\pi \sim \rho$ then $p(A_{\pi})= p(A_{\rho})$. In other words, the projective spectrum $p(A_\pi)$ is also an invariant of $\pi$ with respect to weak equivalence. However, the converse of Lemma \ref{weakeq} is not true.

\begin{example}\label{irrin}
We now consider the group $G=GL_3(\mathbb{Z}/3\mathbb{Z})$ which admits the presentation
\[G=\langle g_1, g_2, g_3 \mid g_1^2=(g_1g_2^{-1})^2=(g_1g_3^{-1})^2=g_2^2g_3g_2^{-1}g_3=g_2g_3^2g_2g_3^{-1}=1\rangle.\]
If we set 
\[A_1=\left(\begin{matrix}
-\frac{1}{\sqrt{2}} & -\frac{1}{2}-\frac{i}{2}\\
 -\frac{1}{2}+\frac{i}{2} & \frac{1}{\sqrt{2}}
 \end{matrix}\right),\ \ 
 A_2=\left(\begin{matrix}
\frac{1}{2}+\frac{i}{2} & \frac{1}{\sqrt{2}}\\
 -\frac{1}{\sqrt{2}} & \frac{1}{2}-\frac{i}{2}
 \end{matrix}\right),\ \ 
A_3=\left(\begin{matrix}
\frac{1}{2}-\frac{i}{2} & \frac{i}{\sqrt{2}}\\
 \frac{i}{\sqrt{2}} & \frac{1}{2}+\frac{i}{2}
\end{matrix}\right),\]
then the maps \[\rho_{\pm}(g_1)=\pm A_1,\ \ \rho_{\pm}(g_2)=A_2,\ \ \rho_{\pm}(g_3)=A_3\]
extend to two unitary representations $\rho_+, \rho_{-}: G\to U_2$ (Klep-Vol\v{c}i\v{c} \cite{KV}). Since $\rho_+(\C[G])=M_2(\C)=\rho_{-}(\C[G])$, both representations are irreducible. Note that for finite groups, weak equivalence coincides with equivalence. If there were a unitary $U\in U_2$ such that $U\rho_+(g)U^*=\rho_{-}(g),\ g\in G$, then the linear transformation $T: M_2(\C)\to M_2(\C)$ defined by $T(M)=UMU^*$ would have eigenvalue $1$ with corresponding eigenvectors $I$, $A_2$ and $A_3$; and it would have eigenvalue $-1$ with corresponding eigenvectors $A_1$, $A_1A_2$ and $A_1A_3$, which is impossible because $\dim M_2(\C)=4$. This verifies the inequivalence of $\rho_+$ and $\rho_{-}$. However, we have \[Q_{\rho_+}(z)=Q_{\rho_{-}}(z)=z_0^2+z_0(z_2+z_3)-z_1^2+z_2^2+z_3^2.\]
\end{example}

Nevetheless, the converse of Lemma \ref{weakeq} holds when $\pi$ is th trivial representation. The method of proof can be found in \cite{BHV} Appendix.

\begin{lemma}\label{spec}
Suppose group $G$ is generated by a finite set $S$ and $\pi$ is a unitary representation. Then $1_G \prec \pi$ if and only if $H_0\subset p(A_\pi)$. 
\end{lemma}
\begin{proof}
The necessity has been observed in Lemma \ref{weakeq}. For the other direction, we assume $S=\{g_1, ..., g_n\}$. Since $H_0\subset p(A_\pi)$, we have 
$(-1, \frac{1}{n}, ..., \frac{1}{n})\in p(A_\pi)$, i.e., $1\in \sigma(M_{\pi})$, where $M_\pi$ is the associated Markov operator. Thus, either $M_{\pi}$ or $M^*_{\pi}$ is not bounded below. Without loss of generality, we assume the former occurs. Then there exists a sequence of unit vectors $\xi_k\in \HH$ such that $\|M_\pi \xi_k-\xi_k\|\to 0$ (and hence $\|M_\pi \xi_k\|\to 1$). It follows that
\begin{align*}
\sum_{i=1}^n\|\pi (g_i) \xi_k - \xi_k\|^2&=n(2-2\re \langle M_\pi \xi_k, \xi_k\rangle)\\
&=n(\|M_\pi \xi_k-\xi_k\|^2+1-\|M_\pi \xi_k\|^2)\to 0,\end{align*}
which implies $\|\pi (g_i) \xi_k - \xi_k\|=\|\xi_k - \pi(g^{-1}_i)\xi_k\|\to 0$ for each $i$. We let $S^{-1}$ stand for the set $\{g_1^{-1}, ..., g_n^{-1}\}$. Then $G=\cup_{m=0}^\infty (S\cup S^{-1})^m$. Since $G$ is countable, its subset $Q$ is compact if and only if it is finite, in which case there exists an integer $M$ such that $Q\subset \hat{S}:=\cup_{m=0}^M (S\cup S^{-1})^m$. Given any $\epsilon>0$, we let $\xi$ be such that \[\|\pi (g_i) \xi - \xi\|=\|\xi - \pi(g^{-1}_i)\xi\|<\frac{\epsilon}{M}, \ \ \ 1\leq i\leq n.\]
Then for each $x=x_1\cdots x_m\in \hat{S}$, where $x_i\in S\cup S^{-1}$, we have
\begin{align*}
&\|\pi (x_1\cdots  x_m) \xi - \xi\|\\ 
=& \|\pi (x_1\cdots x_m) \xi - \pi (x_1\cdots x_{m-1}) \xi +\cdots +\pi (x_1 x_2) \xi - \pi (x_1) \xi + \pi (x_1) \xi - \xi\|\\
\le& \|\pi(x_1\cdots x_{m-1})( \pi (x_m) \xi - \xi)\| +\cdots+ \| \pi(x_1)(\pi (x_2) \xi - \xi)\| + \|\pi (x_1) \xi - \xi\|\\
=& \|\pi (x_m) \xi - \xi\| +\cdots+ \|\pi (x_2) \xi - \xi\| + \|\pi (x_1) \xi - \xi\|<m \frac{\epsilon}{M} < \epsilon.
\end{align*}
It follows that $|1-\langle \pi(x)\xi, \xi\rangle|=|\langle \pi(x)\xi-\xi, \xi\rangle \leq  \|\pi(x)\xi-\xi\|<\epsilon$,
which shows that $1_G\prec \pi$.
\end{proof}

Theorem \ref{main1} is a consequence of Theorem \ref{HR} and Lemma \ref{spec}. In Example \ref{specD}, the slice of $p(A_\lb)$ corresponding to $x=1$ is the union
\[\{z_0+z_1+z_2=0\}\cup \{z_0-z_1-z_2=0\}.\] This reflects the fact that $D_\infty$ is amenable. The following case is more interesting.
\begin{example}
We take another look at Example \ref{F2} regarding the free group. The point $(1, -\frac{1}{2}, -\frac{1}{2})$ is in the hyperplane $H_0\subset \mathbb{P}^2$ but not in $R_0$. Hence $p(A_{\lb})$ does not contain $H_0$. In view of Theorem \ref{main1}, this conforms with the nonamenability of $F_2$. On the other hand, if $\xi_1, \xi_2, \xi_3$ are the distinct roots of equation $\xi^3=1$, then $(\xi_1, \xi_2, \xi_3)\in H_0\cap p(A_{\lb})$. Therefore, $H_0\not\subset p(A_{\lb})$ but $H_0\cap p(A_{\lb})\neq \{0\}$. 
\end{example}

\subsection{Haagerup property and Kazhdan's Property (T)}

The two properties of groups have been extensively studied (\cite{BHV,CCJ}). This section only gives a multivariable interpretation of the properties based on Lemma \ref{spec}, short of an in-depth investigation. However, a new perspective usually brings about a new approach. 

\begin{definition} A unitary representation $(\pi, {\mathcal H})$ of a locally compact group $G$ is said to be of $C_0$ if for every $\epsilon>0$ and every $x, y\in {\mathcal H}$ the set 
\[Q(x,y;\epsilon):=\{g\in G\mid |\langle \pi(g)x, y\rangle|\geq \epsilon\}\] is compact.
\end{definition}

Clearly, for a compact group, every representation is of $C_0$. Hence this definition is meaningful only for noncompact groups $G$. If a representation $(\pi, \HH)$ contains $1_G$, then it possesses an invariant vector $x\in \HH$ in the sense that $\pi(g)x=x, g\in G$. In this case $Q(x,y;\epsilon)=G$ whenever $0<\epsilon<|\langle x, y \rangle|$, and hence $\pi$ is not of $C_0$. The weak containment of $1_G$ motivates the following definitions of the two properties.
\begin{definition}
A locally compact group $G$ is said to have {\em Haagerup property} or ({\em a-T-menability}) if there exists a $C_0$-representation $\pi$ of $G$ such that $1_G\prec \pi$. 
\end{definition}

\begin{definition}
A topological group $G$ is said to have Kazhdan's Property (T) if for every unitary representation $\pi$ of $G$, the weak containment $1_G\prec \pi$ implies the containment $1_G< \pi$.
\end{definition}
Since the left regular representation $\lb_G$ is of $C_0$ (\cite{Di69}), it follows from Theorem \ref{HR} that every amenable group has Haagerup property. If $G$ is compact, then $1_G\prec \pi$ implies $1_G< \pi$. Thus every compact group has Kazhdan's Property (T). On the other hand, no locally compact but noncompact amenable group $G$ has Kazhdan's Property (T) because otherwise Theorem \ref{HR} would give $1_G<\lb$, i.e., the regular representation $\lb$ has an invariant vector $f\in L^2(G)$, or in other words, $f$ is a nontrivial constant function in $L^2(G)$. This would imply that $G$ is compact which is a contradiction. In light of Lemma \ref{spec}, the two properties have the following multivariable description.
\begin{proposition}\label{Kazhdan}
Assume group $G$ is finitely generated. Then the following statements hold.

(a) $G$ has Haagerup property if and only if there exists a $C_0$-representation $\pi$ of $G$ such that $H_0\subset p(A_{\pi})$.

(b) $G$ has Kazhdan's property (T) if and only if for any unitary representation $\pi$ the inclusion $H_0\subset p(A_{\pi})$ implies $1_G< \pi$. 
\end{proposition}

Theorem \ref{main1} and Proposition \ref{Kazhdan} are clear evidences that projective spectrum is able to capture some intrinsic properties of a group.

\section{Self-similarity and Julia Set}

An important line of study on the joint spectrum of groups was carried out by Grigorchuk and his collaborators on self-similar groups, for instance in \cite{BG, GN07, GS}. Roughly speaking, such groups have a measure-preserving action on the rooted binary tree that duplicates at each subtree.
A famous example is the group of intermediate growth (\cite{Gr80,Gr83}). Other well-known examples include $D_\infty$, Basilica group, lamplighter group, etc. We refer the readers to \cite{GNS} and \cite{Ne05} for details. This section aims to reveal a close link between the projective spectrum of $D_\infty$ and the Julia set of a rational map due to the self-similarity. 

The following figure shows the first three layers of the rooted binary tree.
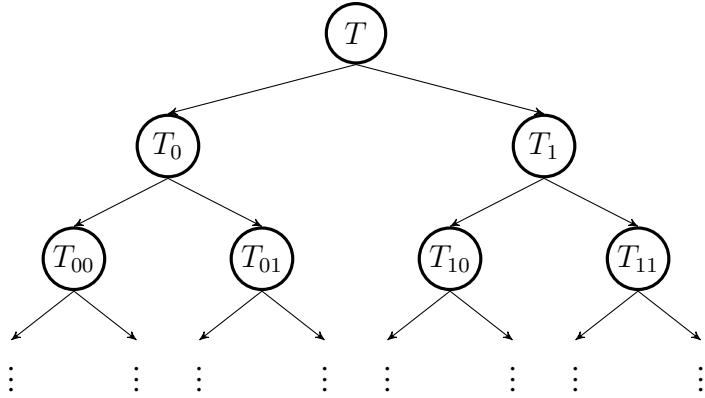
\begin{figure}[h]
\begin{tikzpicture}[->,>=stealth',level/.style={sibling distance = 5cm/#1,level distance = 1.5cm}]
\node [arn_r] {$T$}
    child{ node [arn_r] {$T_0$}
            child{ node [arn_r] {$T_{00}$}
                    child{ node {\vdots}}
                    child{ node {\vdots}}}
            child{ node [arn_r] {$T_{01}$}
                    child{ node {\vdots}}
                    child{ node {\vdots}}}
           }
    child{ node [arn_r] {$T_1$}
            child{ node [arn_r] {$T_{10}$}
                     child{ node {\vdots}}
                     child{ node {\vdots}}}
            child{ node [arn_r] {$T_{11}$}
                      child{ node {\vdots}}
                      child{ node {\vdots}}}
	   }
;

\end{tikzpicture}
\caption{A rooted binary tree}
\label{eq:tree}
\end{figure}

Clearly, the tree $T$ consists of two subtrees $T_0$ and $T_1$, each of which also consists of two subtrees: $T_{00}$ and $T_{01}$, and $T_{10}$ and $T_{11}$, respectively, etc. The boundary $\partial T$ of the tree $T$ is the collection of all infinite sequences of directed arrows from the vertex $T$ down the tree. The uniform Bernoulli measure $\mu$ on $\partial T$ is defined by $\mu(\partial T_{i_1\cdots i_p})=\frac{1}{2^p}$, where $i_k\in \{0, 1\}$ for each $k$. In other words, the measure $\mu$ distributes evenly on the subtrees at every level. Since every element in $\partial T$ corresponds to an infinite sequence of directed arrows in $T$, it corresponds to a sequence of $0$s and $1$s. Hence there is a natural bijection from $\partial T$ to the interval $[0, 1]$ (expressed in binary numbers). And this bijection also identifies the measure $\mu$ with the Lebesgue measure on $[0, 1]$. Define the Hilbert space ${\mathcal H}=L^2(\partial T,\mu)$. Let $\mu_i=2\mu,\ i=0, 1$ be the normalized restrictions of $\mu$ on the boundary of the subtrees $ \partial T_0,\ \partial T_1$, and define ${\mathcal H}_i=L^2(\partial T_i,\ \mu_i),$ $i=0, 1$. Then each ${\mathcal H}_i$ can be identified with ${\mathcal H}$ and hence ${\mathcal H}={\mathcal H}_0\oplus {\mathcal H}_1$ can be identified with ${\mathcal H}\oplus {\mathcal H}$ by a unitary $W$. 

\subsection{Self-similar Representation of $D_\infty$}

The following abstract definition reflects the nature of a self-similar Koopman representation.
\begin{definition}
Given an integer $d\geq 2$, a unitary representationr $(\pi, \HH)$ of a group $G$ is said to be $d$-similar if there exists a unitary operator $W: {\mathcal H}\to {\mathcal H}^d$ such that for every $g\in G$ the $d\times d$ block matrix $\hat{\pi}(g)=W\pi(g)W^*$ has all of its entries either equal to $0$ or of the form $\pi(x), x\in G$.
\end{definition}
\noindent In this case, it is clear that $\hat{\pi}$ is a unitary representationr of $G$ on ${\mathcal H}^d$. Since $\hat{\pi}(g)$ itself and each of its nonzero entries are unitaries, every row or column of $\hat{\pi}(g)$ has precisely one nonzero entry.

The Koopman representation $\pi$ of $D_{\infty}=\langle a, t\mid a^2=t^2=1\rangle$ on the tree $T$ is realized by the following self-similar action: $a$ swaps $T_0$ and $T_1$; $t$ acts on  $T_{0}$ like $a$ acting on $T$, while it acts on $T_1$ like what it does on $T$. Thus, if $w$ is a string of ``0''s and ``1''s (could be empty), then the actions can be described by
\[a(0w)=1w,\ \ a(1w)=0w;\ \ \ t(0w)=0a(w),\ \ t(1w)=1t(w).\]
This action can also be described by the 
automaton in Figure 2, where $a$ and $t$ satisfy the recursive relation
$a=\sigma, t\cong a\oplus t$.
\begin{figure}[h]
\begin{tikzpicture}[>=stealth',shorten >=1pt,auto,node distance=3cm]
 \node[state,label=below:$t$] (S) {$e$};
\node[state,label=below:$a$]         (T) [right of=S] {$\sigma$};
\node[state,label=below:$id$]         (U) [right of=T] {$e$};
\path[->] (S)  edge [loop above] node {1} (S);
 \path[->]            (S) edge              node {0} (T);
\path[->] (T)       edge              node {0/1} (U);
\path[->] (U) edge [loop above] node {0,1} (U);
\end{tikzpicture}
\caption{Automaton of the group $D_\infty$}
\label{fig: Aut}
\end{figure}
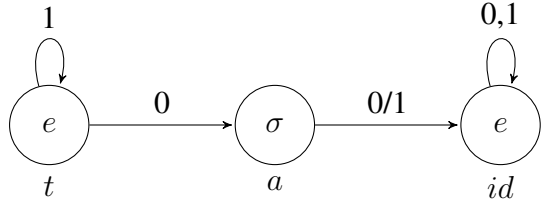
Here, ``$0/1$'' means: given input $0$, the output is $1$; and given input $1$, the output is $0$. Thus, the Koopman representation of $D_\infty$ on $\HH=L^2(\partial T, \mu)$ is $2$-similar, and the identification $W:\HH=\HH_0\oplus \HH_1\to \HH\oplus \HH$ mentioned earlier gives rise to the unitary equivalence
\begin{equation}\label{krepD}
\hat{\pi} (a) \cong
\begin{bmatrix}
0 & I \\
I & 0
\end{bmatrix}, \ \ \
\hat{\pi} (t) \cong
\begin{bmatrix}
\pi (a)  & 0 \\
0 & \pi (t)
\end{bmatrix}.\end{equation}
It is shown in \cite{GY17} that the Koopman representation $\pi$ of $D_\infty$ is weakly equivalent to the regular representationr $\lb$. Thus the following fact is an immediate consequence of Example \ref{specD} and Lemma \ref{weakeq}.
\begin{corollary}\label{specK}
With respect to the Koopman representation $\pi$ of $D_\infty$, we have 
\[p(A_\pi)=\bigcup_{-1\leq x\leq 1}\{z\in {\mathbb P}^2\mid z_0^2-z_1^2-z_2^2-2z_1z_2x=0\}.\]
\end{corollary}

\subsection{The Renormalization Map}
Based on (\ref{krepD}), we have
\begin{equation}\label{block}
A_{\pi}(z) = z_0 +z_1 \pi (a) +z_2 \pi (t) \cong  A_{\hat{\pi}}(z)= \begin{pmatrix}
z_0+z_2 \pi(a) &	 z_1 \\
z_1  &	 z_0+z_2\pi(t)
\end{pmatrix}.
\end{equation}
Therefore, $A_{\pi}(z)$ is invertible if and only if the $2\times 2$ block matrix on the righ-hand side of (\ref{block}) is invertible. For convenience, we often shall write $\pi(g)$ simply as $g$ in the subsequent computations in this section. In the case $z_0^2 \neq z_2^2$, the pencil $z_0+z_2a$ is invertible and its inverse is $(z_0-z_2 a)(z_0^2-z_2 ^2)^{-1}$. In this case, the block matrix $A_{\hat{\pi}}(z)$ is invertible if and only if the Schur complement $z_0+z_2t-z_1 ^2(z_0-z_2a)(z_0^2-z_2^2)^{-1}$ is invertible, or if and only if the rational pencil
\[ \displaystyle \frac{z_0(z_0^2-z_1^2-z_2^2)}{z_0^2-z_2^2}+\frac{z_1 ^2z_2}{z_0^2-z_2^2}a +  z_2 t\] 
is invertible. This gives rise to the following polynomial map:
\begin{align}
\label{eq: F}
F(z_0,z_1,z_2) &= \left( z_0(z_0^2-z_1^2-z_2^2),  z_1^2 z_2,  z_2(z_0^2-z_2^2) \right)\\
:&=\left(F_0(z), F_1(z), F_2(z)\right).\nonumber
\end{align}  
 Since $F$ is homogeneous of degree $2$, it induces a map in the projective space ${\mathbb P}^2$. For convenience, we shall denote this map also by $F$, namely,
\begin{equation}
\label{eq: FP}
F([z_0: z_1: z_2]) = [ z_0(z_0^2-z_1^2-z_2^2): z_1^2 z_2:  z_2(z_0^2-z_2^2)],\ \ \ z\in {\mathbb P}^2.
\end{equation}  
We should be aware that the map $F$ is not well-defined on ${\mathbb P}^2$ at the common zeros of $F_0, F_1$ and $F_2$ because ${\mathbb P}^2$ contains no origin. For $k=0, 1, ...$, we write the $k$th iteration of $F$ as $F^k([z_0: z_1: z_2])=[F^k_0: F^k_1: F^k_2]$ and set
 \[I_k=\bigcup_{m=1}^k\{z\in {\mathbb P}^2\mid F_j^{m}(z)=0, j=0, 1, 2\}.\] It is clear that $I_{k-1}\subset I_{k}$ for each $k$. The set $I_k$ is called the {\em indeterminacy set} of $F^k$, and the closure $E=\overline{\cup_{k=1}^{\infty}I_n}$ is called the extended indeterminacy set of $F$.
\begin{example}
To determine the indeterminacy set $I_1$, we solve the system of equations $F_0=F_1=F_2=0$ in two cases: $z_1=0$ or $z_2=0$, and easily obtain
\begin{equation}\label{prop: orbits}
I_1= \{ [\pm1:1:0], [0:1:0], [\pm 1: 0:1] \}.
\end{equation}
Observe that $I_1$ reflects the spectral property $\sigma(\pi(a))=\sigma(\pi(t))=\{\pm 1\}$. 
\end{example}
\noindent A careful review of the arguments leading to (\ref{eq: F}) shows the following fact.
\begin{lemma}
 $F$ maps $p(A_\pi)\setminus I_1$ into $p(A_\pi)$. 
 \end{lemma}

With a bit more efforts, $I_2$ can be determined as well. However, it is not feasible to fully determine $I_k$ when $k$ gets larger, not mentioning $I_\infty(F)$ or the extended indeterminacy set $E$. This fact adds difficulty to the study of $F's$ dynamical properties. To simplify the situation, we consider the map $\tau: {\mathbb P}^2\to \hat{\C}$ defined in (\ref{eq: tauhat}). Then, according to Corollary \ref{specK}, $\tau(z)\in [-1, 1]$ if and only if $z\in p(A_\pi)$. This fact, in particular, implies that $p^c(A_\pi)=\tau^{-1}(\C\setminus [-1, 1])$ is a dense open subset of ${\mathbb P}^2$. Using the function $\tau$, we can write
\[F(z)=[2\T(z)z_0z_1z_2:z_1^2z_2:z_1z_2(2\T(z) z_2 +z_1)], \ \ z\in {\mathbb P}^2\setminus I_1.\]
This indicates that the complications of the indeterminacy sets $I_k$ is largely due to the common factor $z_1z_2$. Therefore, in order to avoid this unnecessary complication, we consider the {\em renormalization map} of $D_\infty$ associated with the Koopman representation $\pi$ as
\begin{align}
F_\pi(z)=[2\T(z)z_0:z_1:2\T(z) z_2 +z_1],\ \ z\in \mathbb{P}^2\setminus I_1(F_\pi),
\label{eq: F w T}
\end{align}
which is display (5.1) in \cite{GY}. Clearly, $F_\pi(z)=F(z)$ whenever $z_1z_2\neq 0$. It is important to observe that, since $\tau$ is homogeneous of degree $0$, the map $F_\pi$ on ${\mathbb P}^2$ is in fact homogeneous of degree $1$. This lends great convenience to the study of $F_\pi$'s dynamical property. Let $I_k(F_\pi), k=0, 1, ...$ be the indeterminacy sets of $F_\pi$ and let $E_{F_\pi}$ be the extended indeterminacy set. 
\begin{lemma}
$E(F_\pi)=I_1(F_\pi)=\{[\pm 1:0:1]\}$.
\end{lemma}
\begin{proof}
It is sufficient to check that $I_2=I_1$. For a point $z\in {\mathbb P}^2$ to be in $I_1(F_\pi)$, we must have $z_1=0$, and $2\T(z)z_0=2\T(z)z_2=0$. Since $z_0$ and $z_0$ are not both $0$, we have $\tau(z)=0$. It follows from the definition (\ref{eq: tauhat}) that $z_0^2=z_2^2$, and therefore $I_1(F_\pi)=\{[\pm 1:0:1]\}$. 

To determine $I_2(F_\pi)$, it remains to find points $z\in {\mathbb P}^2$ such that $F_\pi(z)\in I_1(F_\pi)$ which means $z_1=0$ and $2\T(z)z_0=\pm 2\T(z)z_2\neq 0$. This implies $z_0=\pm z_2\neq 0$ and hence $z\in I_1(F_\pi)$.
\end{proof} 
\noindent Observe that $E(F_\pi)\subset p(A_\pi)$.

\subsection{The Julia Set of $F_\pi$}

For a rational map $H: \mathbb{P}^n\to \mathbb{P}^n$, the notions of Fatou set and Julia set is defined as follows (\cite{Forn}).

\begin{definition}
  A point $p \in \mathbb{P}^n\setminus E(H)$ is said to be a {\em Fatou point} of $H$ if it has a neighborhood $U\subset \mathbb{P}^n$ on which the sequence of iterations $\{H^{k}\mid k=1, 2, ...\}$ is a normal family. The {\em Fatou set} ${\mathcal F}(H)$ is the set of Fatou points of $H$, and the Julia set ${\mathcal J}(H)$ is the complement $\mathbb{P}^n\setminus {\mathcal F}(H)$.
\end{definition}
Clearly, the extended indeterminacy set $E(H)$ is a subset of ${\mathcal F}(H)$. The following one variable example is crucial for the subsequent discussion.
\begin{example}\label{Tch}
We identify $\mathbb{P}$ with the extended complex plane $\hat{\C}$. Consider the Tchebyshev polynomial $T(x)=2x^2-1, x\in \hat{\C}$. It is known, for instance see \cite{Be}, that its Julia set ${\mathcal J}(T)=[-1, 1]$. Moreover, the iteration sequence $\{T^n\}$ converges to $\infty$ uniformly on every compact subset in ${\mathcal F}(T)$. 
\end{example}
Somewhat surprisingly, the Julia set of the map $F$ in (\ref{eq: F}) was shown to be closely related to the projective spectrum: ${\mathcal J}(F)=p(A_\pi)\cup E$. This is Theorem 5.11 in \cite{GY}. However, this result is not entirely satisfactory due to the lack of a clear picture of $E$. Since the renormalization map $F_\pi$ has a much simpler extended indeterminacy set, it makes one wonder whether a cleaner theorem holds for its Julia set. It is indeed the case.
\begin{theorem}\label{Julia}
${\mathcal J}(F_\pi)=p(A_\pi)$.
\end{theorem}
\noindent The proof follows the same line as that in \cite{GY}, and only some small modifications are needed to suit the change from $F$ to $F_\pi$. For the readers' convenience, we include all the necessary steps here but leave out some details that can be found in \cite{GY}. We start with the following lemma.
\begin{lemma}\label{prop: pre-diagram}
For $k \geq 1$ the following diagram is commutative: 
$$\begin{tikzcd}
    {\mathbb{P}^2 \setminus E(F_\pi)} \arrow{r}{F_\pi^{k}}  \arrow{d}{\tau}
    &  \mathbb{P}^2\setminus E(F_\pi) \arrow{d}{\tau} \\
    {\hat{\mathbb{C}}} \arrow{r}{T^{k}} 
    & {\hat{\mathbb{C}}}.
\end{tikzcd}$$
\end{lemma}
\begin{proof}
It is sufficient to check that $\tau(F_\pi(z))=T(\tau(z)), z\notin E(F_\pi)$, and it can be easily done by a direct computation.
\end{proof}
\noindent This connection is crucial because the iterations of $F_\pi$ can now be studied through that of $T$, and thus Example \ref{Tch} can be used to study the dynamical property of $F_\pi$. For $n=1, 2, ...,$ a proof by induction gives \cite{GY} Lemma 5.8, namely,
\begin{align}\label{Fnz}
&F_\pi^{n}(z)\nonumber \\
=&\bigg[2^n z_0 \prod_{k=0}^{n-1} T^{k} ( \T ): z_1: 2^n z_2 \prod_{k=0}^{n-1} T^{\circ k} ( \T ) + z_1 \bigg(1 + \sum_{k= 1} ^{n-1} 2^k \prod_{i=1}^{k} T^{(n-i)} ( \T )\bigg)  \bigg] .
\end{align}
To simplify subsequent calculations, we define functions 
\begin{align}\label{pn}
p_n(z)=2^n \prod_{k=0}^{n-1} T^{k} ( \T(z) ), \ \ z\in {\mathbb P}^2, n=1, 2, ...
\end{align}
Observe that if $z\in p^c(A_\pi)$, then $\T(z)\notin [-1, 1]$, and Example \ref{Tch} implies $T^{k} ( \T(z) )\neq 0$ for every $k\geq 0$. Thus, as in \cite{GY} display (5.2), we can set
\begin{equation}
\label{eq: fn}
f_n(z) = \sum_{j=1}^{n} \frac{1}{p_j(z)}, \ n \geq 2, z\in p^c(A_\pi).
\end{equation} 
Since $T(x)$ is holomorphic on $\hat{\C}$ and $\tau$ is holomorphic on $p^c(A_\pi)$, the function $f_n$ is holomorphic on $p^c(A_\pi)$. Then, $F_\pi^{n}(z)$ can be simplified as 
\begin{align}\label{lem: Fn}
F_\pi^{n} (z)=\left[z_0: \frac{z_1}{p_n(z)}: z_2+ z_1f_n(z)\right],\ \ z\in p^c(A_\pi),
\end{align}
which is Lemma 5.8 in \cite{GY}. And the following lemma holds as well.
\begin{lemma}  $p^c(A_\pi) \subset {\mathcal F}(F_\pi) $. 
\label{thm: fatou}
\end{lemma}
\begin{proof}
Since $\T(z)$ is holomorphic on $p^c(A_\pi)$, for every compact subset $K\subset p^c(A_\pi)$ the image $\T(K)$ is compact in $\hat{\C}\setminus [-1, 1]$. Thus $\{T^{k}(\T(z))\}$ converges uniformly to $\infty$ on $K$ by Example \ref{Tch}. The definition (\ref{pn}) thus permits the existence of a $N\in \mathbb{N}$ such that $\frac{1}{|p_n(z)|}\leq 2^{-n}$ for every $n\geq N$. This implies that the series 
\begin{equation}
\label{eq: fz}
f(z) = \lim_{n \to \infty} f_n(z)= \frac{1}{p_1(z)}+ \frac{1}{p_2(z)}+\cdots
\end{equation}
converges uniformly on $K$. Equation (\ref{lem: Fn}) then implies that $F_\pi^{n} (z)$ converges normally on $p^c(A_\pi)$ to the map 
\begin{equation}\label{Fstar}
F_\ast ([z_0:z_1:z_2])=[z_0:0:z_2 +z_1 f(z)],\ \ \ z\in p^c(A_\pi).
\end{equation}
\end{proof}
Lemma \ref{thm: fatou} implies the inclusion ${\mathcal J}(F_\pi)\subset p(A_\pi)$. The other direction
$p(A_\pi)\subset {\mathcal J}(F_\pi)$ is already shown in the proof of \cite{GY} Theorem 5.11. But here we shall give a simpler proof using the density of $p^c(A_\pi)$ in $\mathbb{P}^2$. Observe that if $\xi \in p(A_\pi)$ then $\T(\xi) \in [-1,1]$, and hence we can write $\tau(\xi)=\cos \theta$, for some $\theta\in [0, \pi]$. Then $T ^{n} (\tau(\xi)) = \cos (2^n \theta), n=1, 2, ...$. Suppose $\xi\in p(A_\pi)$ is such that $\frac{\theta}{\pi}$ is non-dyadic, i.e., $2^n\frac{\theta}{\pi}\notin {\mathbb Z}$ for any integer $n\geq 0$. Then, in light of Corollary 4.2, we have $\xi_1\neq 0$. Moreover, for every $n$,
\begin{equation} 
p_n(\xi)=2^n \prod_{k=0}^{n-1} \cos(2^k \theta )= \frac{\sin (2^{n}\theta)}{\sin \theta}\neq 0,
\end{equation}
and it follows from (\ref{Fnz}) that
\[ F_\pi^{n}(\xi) = \bigg[\xi_0 : \xi_1\frac{\sin \theta}{\sin (2^{n}\theta)}: \xi_2  + \xi_1\sin\theta \sum_{k=1}^{n}(\sin(2^k\theta))^{-1}\bigg]. \]
If $\xi$ were a Fatou point, then there would exist a path-connected neighborhood $V$ of $\xi$ and a subsequence $\{F_\pi^{n_k}\}$ that converges normally to a holomorphic function $\hat{F_*}$ on $V$. Since $p^c(A_\pi)$ is dense in $\mathbb{P}^2$, the limit $\hat{F_*}$ must be an holomorphic extension of the function $F_*$ in (\ref{Fstar}) from $V\cap p^c(A_\pi)$ to $V$. But due to the fact $|\frac{\sin \theta}{\sin (2^{n}\theta)}|\geq \sin \theta>0$ for all $n$, $\hat{F_*}(z)$ is not continuous at $\xi$. This is a contradiction. Theorem \ref{Julia} is thus established. 

Using the same method as in the proof of \cite{GY} Theorem 5.14, the limit function $f$ in (4.11) can be explicitly determined as
\begin{equation}\label{Lf}
f(z)=\tau(z)-\sqrt{\tau^2(z)-1},\ \ \ z\in p^c(A_\pi),
\end{equation}
\noindent where $\tau$ is as defined in (\ref{eq: tauhat}). 

Theorem \ref{Julia}, together with formula (4.12) and (\ref{Lf}), present a rare case where the Julia set of a nontrivial multivariable map and the limit of its iteration sequence can both be explicitly computed. What is more interesting is the role played by self-similar group representation in this case. A recent exposition on this subject can be found in Dang-Grigorchuk-Lyubich \cite{DGL}. In Zu-Yang-Lu \cite{ZYL}, similar results are obtained for the lamplighter group, though in a more complicated form. These facts make one wonder whether the same is true for other self-similar groups, such as the Basilica group and the Grigorchuk group of intermediate growth.

\end{document}